\documentclass[12pt,reqno]{amsart}
\usepackage{amssymb,amsmath,amsthm,hyperref}
\newcommand{\sg}{\textnormal{sg}}

\newtheorem{theorem}{Theorem}
\newtheorem{lemma}[theorem]{Lemma}
\newtheorem{corollary}[theorem]{Corollary}
\newtheorem{proposition}[theorem]{Proposition}

\theoremstyle{remark}
\newtheorem*{remark}{Remark}
\theoremstyle{definition}

\newtheorem{definition}[theorem]{Definition}
\newtheorem{example}{Example}

\numberwithin{theorem}{section} \numberwithin{equation}{section}
\numberwithin{example}{section}

\title[Partial theta functions and Appell--Lerch sums]{On the dual nature of partial theta functions\\ and Appell--Lerch sums}

\author{Eric Mortenson}

\begin{document}

\date{12 June 2014}



\begin{abstract}
In recent work, Hickerson and the author demonstrated that it is useful to think of Appell--Lerch sums as partial theta functions.  This notion can be used to relate identities involving partial theta functions with identities involving Appell--Lerch sums.  In this sense, Appell--Lerch sums and partial theta functions appear to be dual to each other.  This duality theory is not unlike that found by Andrews between various sets of identities of Rogers-Ramanujan type with respect to Baxter's solution to the hard hexagon model of statistical mechanics.  As an application we construct bilateral $q$-series with mixed mock modular behaviour.  In subsequent work we see our bilateral series are well-suited for computing radial limits of Ramanujan's mock theta functions.
\end{abstract}

\address{School of Mathematics and Physics, University of Queensland, Brisbane 4072}
\address{Max-Planck-Institut f\"ur Mathematik, Vivatsgasse 7, 53111 Bonn, Germany}
\email{etmortenson@gmail.com}
\maketitle
\setcounter{section}{-1}

\section{Notation}\label{section:notation}

 Let $q$ be a nonzero complex number with $|q|<1$ and define $\mathbb{C}^*:=\mathbb{C}-\{0\}$.  Recall
\begin{gather}
(x)_n=(x;q)_n:=\prod_{i=0}^{n-1}(1-q^ix), \ \ (x)_{\infty}=(x;q)_{\infty}:=\prod_{i\ge 0}(1-q^ix),\notag \\
{\text{and }} \ \ j(x;q):=(x)_{\infty}(q/x)_{\infty}(q)_{\infty}=\sum_{n=-\infty}^{\infty}(-1)^nq^{\binom{n}{2}}x^n,\label{equation:theta-def}
\end{gather}
where in the last line the equivalence of product and sum follows from Jacobi's triple product identity.    Here $a$ and $m$ are integers with $m$ positive.  Define
\begin{gather*}
J_{a,m}:=j(q^a;q^m), \ \ J_m:=J_{m,3m}=\prod_{i\ge 1}(1-q^{mi}), \ {\text{and }}\overline{J}_{a,m}:=j(-q^a;q^m).
\end{gather*}
We will use the following definition of an Appell--Lerch sum \cite{HM}:
\begin{equation}
m(x,q,z):=\frac{1}{j(z;q)}\sum_{r=-\infty}^{\infty}\frac{(-1)^rq^{\binom{r}{2}}z^r}{1-q^{r-1}xz}.\label{equation:mdef-eq}
\end{equation}
The symbol $\sum^{\ast}$ indicates convergence problems, so care should be taken.
\section{Introduction}

In his last letter to Hardy, Ramanujan gave a list of seventeen functions which he called ``mock theta functions.''   \cite[p. xxxi]{R1}: ``{\em I am extremely sorry for not writing you a single letter up to now$\dots$I discovered very interesting functions recently which I call `Mock' $\vartheta$-functions.  Unlike the `False' $\vartheta$-theta functions (studied partially by Prof. Rogers in his interesting paper) they enter mathematics as beautifully as the ordinary theta functions$\dots$}''  Each mock theta function was defined by Ramanujan as a $q$-series convergent for $|q|<1$.  He stated that they have certain asymptotic properties as $q$ approaches a root of unity, similar to the properties of ordinary theta functions, but that they are not theta functions.  He also stated several identities relating some of the mock theta functions to each other.  Later, many more mock theta function identities were found in the Lost Notebook \cite{RLN}.

Numerous entries in the Lost Notebook expand Eulerian forms ($q$-hypergeometric series) in terms of theta functions (Rogers-Ramanujan type identities), Appell--Lerch sums (mock theta functions), or partial theta functions.   Although partial theta functions are arguably the least understood they do play significant roles in areas outside of number theory such as quantum invariants of $3$-manifolds \cite{LZ}.  Appell--Lerch sums also appear naturally in the context of black hole physics \cite{DMZ}.  One wants to understand the various types of Eulerian forms and how they relate to each other.  In this direction, Andrews \cite{A2} has recently produced $q$-hypergeometric formulas which simultaneously prove mock theta function identities and Rogers-Ramanujan type identities. 

 We note that a {\em false theta function} is simply a theta series (\ref{equation:theta-def}) but with the wrong signs; whereas a {\em partial theta function} is half of a theta series.  In some cases one can write one in terms of the other, see Example \ref{example:ex-1}.

In recent work \cite{HM}, Hickerson and the author introduced a connection between partial theta functions and Appell--Lerch sums---the building blocks of the classical mock theta functions---to obtain a general formula that expands a certain family of indefinite theta series in terms of Appell--Lerch sums and theta functions.  For a discussion of the heuristic see \cite[Section $2$]{HM}.
The techniques of \cite{HM} can be used to relate identities involving partial theta functions with identities involving Appell--Lerch sums.  In this sense, Appell--Lerch sums and partial theta functions appear to be dual to each other. 

Much can be learned from the transformation $q\rightarrow q^{-1}$.  For example, it is useful in problems in which the theory of partitions is applied to statistical mechanics \cite{A3, AB, ABF}.  Although $q\rightarrow q^{-1}$ may make sense in a $q$-hypergeometric series, it will not for a product.  In the duality theory initiated by Andrews \cite{A3} for various sets of identities of Rogers-Ramanjuan type with respect to Baxter's solution to the hard hexagon model of statistical mechanics \cite{Bax}, Andrews used finite versions, i.e. polynomial or rational function identities which converge to infinite $q$-series in the limit.  To translate between identities of Regions I and IV \cite{A3, Bax}, he used identities  which have origins in work of Schur \cite{A0,Sch}, 
 \cite[4.1, 4.2]{A3}:
\begin{align}
\sum_{j=0}^{\infty}q^{j^2}\left [\begin{matrix} N-j \\ j   \end{matrix} \right]_q&=\sum_{\lambda=-\infty}^{\infty}(-1)^{\lambda}q^{\lambda(5\lambda+1)/2}\left [\begin{matrix} N \\ \lfloor \frac{N-5\lambda}{2}   \rfloor  \end{matrix} \right]_q,\label{equation:Andrews-4.1}\\
\sum_{j=0}^{\infty}q^{j^2+j}\left [\begin{matrix} N-j \\ j   \end{matrix} \right]_q&=\sum_{\lambda=-\infty}^{\infty}(-1)^{\lambda}q^{\lambda(5\lambda-3)/2}\left [\begin{matrix} N+1 \\ \lfloor  \frac{N+1-5\lambda}{2}   \rfloor  +1 \end{matrix} \right]_q,\label{equation:Andrews-4.2}
\end{align}
where $\left [\begin{matrix} n \\ m   \end{matrix}\right ] _q=(q;q)_n/(q;q)_m(q;q)_{n-m}$ for $0\le m\le n$ and zero otherwise, and $\lfloor \cdot \rfloor$ is the greatest integer function.  When $N\rightarrow \infty$, the two identities become the Rogers-Ramanjuan identities, i.e. Region I.  To obtain the identities of Region IV, replace $N$ with $2N+a$, $a\in\{0,1\}$,  then replace $q$ with $q^{-1}$, multiply by the appropriate power of $q$ to have a polynomial, and then let $N\rightarrow \infty$.

For our dual notion we do not necessarily need finite versions to make sense of the side which is not a $q$-hypergeometric series.  To convert between identities expressing Eulerian forms in terms of  partial theta functions and identities expressing Eulerian forms in terms of Appell--Lerch sums (and vice versa): 

\begin{itemize}
\item[(i)] make the subsitution $q\rightarrow q^{-1}$ in the Eulerian form. 

\item[(ii)] Let $\rho=q^{-1}$ and use the identity $(a;\rho)_{n}=(a^{-1};q)_n(-a)^{n}\rho^{\binom{n}{2}}.$

\item[(iii)] Use the heuristic $m(x,q,z)\sim \sum_{r\ge 0}(-1)^rx^rq^{-\binom{r+1}{2}}$ where `$\sim$' means up to the addition of a theta function (see \cite[Section 2]{HM}).

\item[(iv)] Use numerical work to determine the remainder, for example one can use Maple to factor resulting $q$-series in terms of theta functions.
\end{itemize}

\begin{example}\label{example:ex-1}Multiple Eulerian forms may have the same Appell--Lerch sum expression for $|q|<1$.  As an example, we recall the second order mock theta function $B(q)$ \cite{HM}:
\begin{align}
B(q)=&\sum_{n= 0}^{\infty}\frac{q^{n}(-q;q^2)_n}{(q;q^2)_{n+1}}=\sum_{n=0}^{\infty}\frac{q^{n^2+n}(-q^2;q^2)_n}{(q;q^2)_{n+1}^2}=-q^{-1}m(1,q^4,q^3)
\end{align}
There is no reason to expect that the above two Eulerian forms will have anything in common with each other after making the substitution $q\rightarrow q^{-1}$.  Using (ii) and the heuristic (iii), the duals are respectively
{\allowdisplaybreaks \begin{align}
& \sum_{n= 0}^{\infty}\frac{(-1)^{n+1}q^{n+1}(-q;q^2)_n}{(q;q^2)_{n+1}}=-q\sum_{n=0}^{\infty}(-1)^{n}q^{2n(n+1)},\label{equation:B-first}\\
 &\sum_{n= 0}^{\infty}\frac{q^{2(n+1)}(-q^2;q^2)_n}{(q;q^2)_{n+1}^2}=-q\sum_{n=0}^{\infty}(-1)^{n}q^{2n(n+1)}
 +q\frac{J_4}{J_{1,2}}\Big (  \sum_{n=0}^{\infty}q^{-n}q^{3n(n+1)}-q\sum_{n=0}^{\infty}q^{n}q^{3n(n+1)}\Big ).\label{equation:B-second}
\end{align}}%
The right-hand side  of (\ref{equation:B-first}) is a {\em partial theta function}, while the right-hand side of (\ref{equation:B-second}) a {\em mixed partial theta function}.  We call the second term on the right of (\ref{equation:B-second}) the {\em mixed term}, where the term in parentheses is a {\em false theta function} which is the sum of two {\em partial theta functions}.  We see the partial theta functions in (\ref{equation:B-first}) and (\ref{equation:B-second}) are equal, but that (\ref{equation:B-second}) has an extra term.  There are many more such examples in which the heuristic  is effective in predicting the partial theta function (non-mixed term), but it is not immediately obvious how the heuristic can predict the mixed term.
\end{example}

How effective is the heuristic in predicting the partial theta functions and Appell--Lerch sums when there are multiple non-mixed terms?  This brings us to two recent mock theta functions of Andrews \cite{A2}:
\begin{example}\label{example:ex-2}
We recall from the Lost Notebook two partial theta functions identities ({Entry $6.5.1$ \cite{ABII}, also \cite[p. 31]{RLN}}):
\begin{align}
\sum_{n=0}^{\infty}\frac{q^n}{(-q)_{2n}}
&=\sum_{n=0}^{\infty}q^{12n^2+n}(1-q^{22n+11})+q\sum_{n=0}^{\infty}q^{12n^2+7n}(1-q^{10n+5})\label{equation:ABII-6.5.1A}
\end{align}
and
\begin{align}
\sum_{n=0}^{\infty}\frac{q^n}{(-q)_{2n+1}}
&=\sum_{n=0}^{\infty}q^{12n^2+5n}(1-q^{14n+7})+q^2\sum_{n=0}^{\infty}q^{12n^2+11n}(1-q^{2n+1}).\label{equation:ABII-6.5.1B}
\end{align}
We first consider (\ref{equation:ABII-6.5.1A}).  For the left-hand side, the substitution $q\rightarrow q^{-1}$ yields
\begin{equation}
\sum_{n=0}^{\infty}\frac{q^{2n^2}}{(-q)_{2n}}=:\overline{\psi}_0(q),
\end{equation}
where $\overline{\psi}_0(q)$ is the third order mock theta function discovered by Andrews \cite[$(1.14)$]{A2}.  From \cite[Theorem $1.5$]{M1} we have
\begin{align}
\overline{\psi}_0(q)&=2-2qg(-q,q^8)-{J_{1,2}\overline{J}_{3,8}}/{J_{2}}\label{equation:Andrews-psi0}\\
&=m(-q^{11},q^{24},q^{4})+m(-q^{11},q^{24},q^{22})+q^{-1}m(-q^5,q^{24},q^4)+q^{-1}m(-q^5,q^{24},q^{10}),\notag
\end{align}
where the last line is new but can be shown using the techniques of \cite{HM}.  We show that the heuristic takes us from the right-hand side of (\ref{equation:ABII-6.5.1A}) to the right-hand side of (\ref{equation:Andrews-psi0}).  We rewrite the right-hand side of (\ref{equation:ABII-6.5.1A}) and then make the substitution $q\rightarrow q^{-1}$:
\begin{align*}
\sum_{n=0}^{\infty}&(q^{-11})^nq^{24\binom{n+1}{2}}-q^{11}\sum_{n=0}^{\infty}(q^{11})^nq^{24\binom{n+1}{2}}+q\sum_{n=0}^{\infty}(q^{-5})^nq^{24\binom{n+1}{2}}
-q^6\sum_{n=0}^{\infty}(q^{5})^nq^{24\binom{n+1}{2}}\\
&\rightarrow  m(-q^{11},q^{24}, *)- q^{-11}m(-q^{-11},q^{24}, *) +q^{-1} m(-q^{5},q^{24}, *) -q^{-6}m(-q^{-5},q^{24}, *),
\end{align*}
where we have followed with the heuristic (iii), and the `$*$' indicates up to the addition of a theta function, see also (\ref{equation:m-change-z}).  Identity (\ref{equation:m-fnq-flip}) then shows  we can view (\ref{equation:Andrews-psi0}) as the dual of (\ref{equation:ABII-6.5.1A}).   For (\ref{equation:ABII-6.5.1B}), the substitution $q\rightarrow q^{-1}$ in the left-hand side yields
\begin{equation}
\sum_{n=0}^{\infty}\frac{q^{2n^2+2n+1}}{(-q)_{2n}}=:q\overline{\psi}_1(q),
\end{equation}
where $\overline{\psi}_1(q)$ is the third order mock theta function discovered by Andrews \cite[$(1.15)$]{A2}.  From \cite[Theorem $1.5$]{M1}, we have
\begin{align}
q\overline{\psi}_1(q)&=2q^3g(-q^3,q^8)+q{J_{1,2}\overline{J}_{1,8}}/{J_{2}}\label{equation:Andrews-psi1}\\
&=m(-q^{7},q^{24},q^{8})+m(-q^{7},q^{24},q^{16})-q^{-2}m(-q^{-1},q^{24},q^{8}) -q^{-3}m(-q^{-1},q^{24},q^{16}).\notag
\end{align}
Arguing as above shows that we can view (\ref{equation:Andrews-psi1}) as the dual of (\ref{equation:ABII-6.5.1B}).
\end{example}

In Section \ref{section:prelim} we recall basic facts about Appell--Lerch sums and Bailey pairs.   In Proposition \ref{proposition:eulerian-mxqz-prop} we rewrite several $q$-hypergeometric series found in the lost notebook in terms of Appell--Lerch sums.  In Section \ref{section:mock}, we prove the duals for two fifth order mock theta functions, the three seventh orders, and the four tenth orders.  We emphasize that all of the duals follow from the same conjugate Bailey pair and that all of the Bailey pairs are from \cite{S, Wa1}. We also prove identities for the duals of the four tenth orders, see (\ref{equation:tenth-dual-I})-(\ref{equation:tenth-dual-IV}).   In the first example we pointed out that it was not immediately obvious how one can obtain the mixed term using the heuristic.   This will be demonstrated in Section \ref{section:duals}, where we state the duals for many mixed partial theta functions in the Lost Notebook and also relate the mixed partial theta functions to mixed mock modular forms.  The relation to mixed mock modular forms will be referred to as a dual of second type.   Our duals of second type are useful for constructing bilateral $q$-series with mixed mock modular behaviour.

In subsequent work \cite{M2}, we see that our bilateral series with mixed mock modular behaviour are well-suited for computing the radial limits of mock theta functions (see [FOR1, FOR2]) and for addressing questions found at the end of [R].  In \cite{M2}, we present five more radial limit results which follow from mixed mock modular bilateral $q$-hypergeometric series.  We also obtain the mixed mock modular bilateral series for a universal mock theta function of Gordon and McIntosh.  The later bilateral series can be used to compute radial limits for many classical second, sixth, eighth, and tenth order mock theta functions.

On a final note, we point out the difference between this paper and \cite{BFR}.  The latter demonstrates families of two-parameter Eulerian forms which agree for $|q|<1$ (with perhaps other additional restrictions) but disagree once one has made $q\rightarrow q^{-1}$ , e.g. the new expressions agree on the partial theta function but disagree on the mixed term, like Example \ref{example:ex-1}.  What we demonstrate here is {\em how} to determine the structure of one type of identity given the structure of the other type.

\section{Preliminaries}\label{section:prelim}

\subsection{Bailey pairs}
A Bailey pair relative to $(a,q)$ is a pair of sequences $(\alpha_n,\beta_n)_{n\ge 0}$ such that
\begin{equation}
\beta_n=\sum_{r=0}^n\frac{\alpha_r}{(aq)_{n+r}(q)_{n-r}}.\label{equation:bp-def}
\end{equation}
A conjugate Bailey pair relative to $(a,q)$ is a pair of sequences  $(\delta_n,\gamma_n)_{n\ge 0}$, where
\begin{equation}
\gamma_n=\sum_{r=n}^{\infty}\frac{\delta_r}{(aq)_{r+n}(q)_{r-n}}.\label{equation:bp-conj-def}
\end{equation}
Given a Bailey pair and a conjugate Bailey pair, we have
\begin{equation}
\sum_{n=0}^{\infty}\beta_n\delta_n=\sum_{n=0}^{\infty}\alpha_n\gamma_n.\label{equation:bp-conj-id}
\end{equation}

\begin{lemma}\label{lemma:important-lemma}  For a Bailey pair $(\alpha_n,\beta_n)_{n\ge 0}$ relative to $(a,q)$,
\begin{equation}
\sum_{n=0}^{\infty}(-1)^nq^{\binom{n+1}{2}}(a)_n\beta_n=\frac{(q)_{\infty}}{(aq)_{\infty}}\sum_{n=0}^{\infty}(-1)^nq^{\binom{n+1}{2}}\frac{(a)_n}{(q)_n}\alpha_n.\label{equation:important-id}
\end{equation}
\end{lemma}
\begin{proof}[Proof of Lemma \ref{lemma:important-lemma}]
Define
\begin{equation}
\delta_n:=\frac{(-1)^nq^{\binom{n+1}{2}}(a)_n}{1-a}.
\end{equation}
Thus
{\allowdisplaybreaks \begin{align*}
\gamma_n&=\sum_{r=n}^{\infty}\frac{1}{(aq)_{r+n}(q)_{r-n}}\cdot \frac{(-1)^rq^{\binom{r+1}{2}}(a)_r}{1-a}\\
&=\sum_{r=0}^{\infty}\frac{(-1)^{r+n}q^{\binom{r+n+1}{2}}(a)_{r+n}}{(a)_{r+2n+1}(q)_{r}}\\
&=\frac{(-1)^nq^{\binom{n+1}{2}}(a)_n}{(a)_{2n+1}}\sum_{r=0}^{\infty}\frac{(-1)^{r}q^{\binom{r}{2}+r(n+1)}(aq^{n})_{r}}{(aq^{2n+1})_{r}(q)_{r}}\\
&=\frac{(-1)^nq^{\binom{n+1}{2}}(a)_n}{(a)_{2n+1}}\cdot{}_1\phi_1\big ( aq^{n};aq^{2n+1};q,q^{n+1}\big )\\
&=(-1)^nq^{\binom{n+1}{2}}\frac{(q)_{\infty}}{(q)_n}\frac{(a)_n}{(aq)_{\infty}}\frac{1}{(1-a)},
\end{align*}}%
where the last line follows from the well-known summation
\begin{equation*}
{}_1\phi_1(a;c;q,c/a)=\frac{(c/a)_{\infty}}{(c)_{\infty}}.\qedhere
\end{equation*}
\end{proof}

\subsection{Properties of the Appell--Lerch sums}\label{section:prop-mxqz}
Everything in this subsection can be found in \cite{HM}.   A simple shift in the summation index of (\ref{equation:mdef-eq}) yields another useful form for $m(x,q,z)$:
\begin{equation}
m(x,q,z)=\frac{-z}{j(z;q)}\sum_{r=-\infty}^{\infty}\frac{(-1)^rq^{\binom{r+1}{2}}z^r}{1-q^{r}xz}.\label{equation:alt-mdef-eq}
\end{equation}

The Appell--Lerch sum $m(x,q,z)$ satisfies several functional equations and identities, which we collect in the form of a proposition:

 \begin{proposition} \cite{Le} \cite[Section $2$]{HM} For generic $x,z,z_0,z_1\in \mathbb{C}^*$
{\allowdisplaybreaks \begin{subequations}
\begin{equation}
m(x,q,z)=m(x,q,qz),\label{equation:m-fnq-z}
\end{equation}
\begin{equation}
m(x,q,z)=x^{-1}m(x^{-1},q,z^{-1}),\label{equation:m-fnq-flip}
\end{equation}
\begin{equation}
m(qx,q,z)=1-xm(x,q,z),\label{equation:m-fnq-x}
\end{equation}
\begin{equation}
m(x,q,z_1)-m(x,q,z_0)=\frac{z_0J_1^3j(z_1/z_0;q)j(xz_0z_1;q)}{j(z_0;q)j(z_1;q)j(xz_0;q)j(xz_1;q)},\label{equation:m-change-z}
\end{equation}
\begin{equation}
m(x,q,z)=m(x,q,x^{-1}z^{-1}).\label{equation:m-fnq-zflip}
\end{equation}
\end{subequations}}
\end{proposition}

We recall the universal mock theta function
\begin{equation}
g(x,q):=x^{-1}\Big ( -1 +\sum_{n=0}^{\infty}\frac{q^{n^2}}{(x)_{n+1}(q/x)_{n}} \Big ),\label{equation:g-def}
\end{equation}
as well as the easily shown
\begin{proposition} \label{proposition:newgid}For $x\ne 0$
\begin{equation}
g(x,q)=\sum_{n=0}^{\infty}\frac{q^{n(n+1)}}{(x)_{n+1}(q/x)_{n+1}}.
\end{equation}
\end{proposition}
\noindent We recall from \cite[Theorem $2.2$]{H1}, \cite[Proposition $3.2$]{HM} an expression relating $g(x,q)$ to $m(x,q,z)$:
\begin{equation}
g(x,q)=-x^{-1}m(q^2x^{-3},q^3,x^2)-x^{-2}m(qx^{-3},q^3,x^2).\label{equation:g-to-m}
\end{equation}
\subsection{Hecke-type double sums}  Here we recall a definition \cite{HM}
\begin{definition} \label{definition:fabc-def}  Let $x,y\in\mathbb{C}^*$ and define $\sg (r):=1$ for $r\ge 0$ and $\sg(r):=-1$ for $r<0$. Then
\begin{equation}
f_{a,b,c}(x,y,q):=\sum_{\substack{\sg (r)=\sg(s)}} \sg(r)(-1)^{r+s}x^ry^sq^{a\binom{r}{2}+brs+c\binom{s}{2}}.\label{equation:fabc-def}\\
\end{equation}
\end{definition}
We give a special case of \cite[Theorem $0.4$ ]{HM} in which $a=b=2$, $c=1$:
\begin{proposition}\label{proposition:f221}We have
\begin{align*}
f_{2,2,1}(x,y,q)&=j(x;q^2)m(-qy/x,q,-1)+j(y;q)m(qx/y^2,q^2,-1)\\
&\ \ \ \ -\frac{1}{\overline{J}_{0,1}\overline{J}_{0,2}}\sum_{d=0}^1\frac{q^{\binom{d+1}{2}}j(q^{d+1}y;q^2)j(-q^{1-d}x/y;q^2)J_2^3j(-q^{2+d}/y;q^2)}{j(-q^{1}x/y^2;q^2)j(q^{d+1}y/x;q^2)}.
\end{align*}
\end{proposition}

\subsection{$q$-hypergeometric series as Appell--Lerch sums}
The first proposition is based on equations of \cite{RLN} many of which were proved in \cite{AM}.
\begin{proposition}\label{proposition:eulerian-mxqz-prop} We have
{\allowdisplaybreaks \begin{align}
(1+x^{-1})\sum_{n= 0}^{\infty}\frac{q^{n+1}(-q)_{2n}}{(qx,q/x;q^2)_{n+1}}&=-m(x,q^2,q)\label{equation:RLNid1}\\
\sum_{n= 0}^{\infty}\frac{(-1)^nq^{n^2}(q;q^2)_{n}}{(-x;q^2)_{n+1}(-q^2/x;q^2)_{n}}&=m(x,q,-1)+\frac{J_{1,2}^2}{2j(-x;q)}\label{equation:RLNid2}\\
&=2m(x,q,-1)-m(x,q,\sqrt{-q/x}) \notag \\
& =m(-qx^2,q^4,-q^{-1})-q^{-1}xm(-q^{-1}x^2,q^4,-q)\notag\\
\sum_{n= 0}^{\infty}{}^{\ast}\frac{(-1)^n(q;q^2)_n}{(-x)_{n+1}(-q/x)_n}&=m(x,q,-1)\label{equation:RLNid3}\\
\Big ( 1+\frac{1}{x}\Big )\sum_{n=0}^{\infty}\frac{(-1)^n(q;q^2)_nq^{(n+1)^2}}{(-xq,-q/x;q^2)_{n+1}}&=m(x,q,-1)-\frac{J_{1,2}^2}{2j(-x;q)}\label{equation:RLNid4}\\
\sum_{n= 0}^{\infty}\frac{(-1)^nq^{2n^2}(q^2;q^4)_{n}}{(-x;q^4)_{n+1}(-q^4/x;q^4)_{n}}&=m(x,q^2,q)+\frac{\overline{J}_{1,4}^2j(-xq^2;q^4)}{j(-x;q^4)j(xq;q^2)}\label{equation:RLNid5}
\end{align}}%
where \cite[Entry $12.3.3$]{ABI}
\begin{equation}
\sum_{n= 0}^{\infty}{}^{\ast}\frac{(-1)^n(q;q^2)_n}{(-x)_{n+1}(-q/x)_n}
:=\frac{1}{\overline{J}_{0,1}}\sum_{n=-\infty}^{\infty}\frac{(1+1/x)q^{n(n+1)/2}}{(1+xq^n)(1+q^n/x)}.\label{equation:sumstar-def}
\end{equation}
\end{proposition}
Although there does not appear to be a way to represent $m(x,q,z)$ as an Eulerian form, one can write $m(x,q,z)$ as a bilateral sum:
\begin{proposition}\label{proposition:bilateral-mxqz-prop} For $a,b\ne 0$,
\begin{align}
\sum_{n=-\infty}^{\infty}&\frac{a^{-n-1}b^{-n}}{(-1/a;q)_{n+1}(-q/b;q)_n}q^{n^2}\\
&=\sum_{n=-\infty}^{\infty}(-aq;q)_{n}(-b;q)_{n+1}q^{n+1}=\frac{(-aq)_{\infty}}{b(q)_{\infty}(-q/b)_{\infty}}j(-b;q)m\big ( a/b,q,-b\big ).\notag
\end{align}
\end{proposition}
\begin{proof} This follows from \cite[p. $15$]{RLN}, \cite[Entry $3.4.7$]{ABII}
\begin{align}
&\sum_{n=0}^{\infty}\frac{a^{-n-1}b^{-n}}{(-1/a;q)_{n+1}(-q/b;q)_n}q^{n^2}\label{equation:3.4.7-ABII}\\
&+\sum_{n=1}^{\infty}(-aq;q)_{n-1}(-b;q)_{n}q^{n}=\frac{(-aq)_{\infty}}{(q)_{\infty}(-q/b)_{\infty}}\Big ( \sum_{n=0}^{\infty}\frac{b^nq^{n(n+1)/2}}{1+aq^n}
+\frac{1}{a}\sum_{n=1}^{\infty}\frac{b^{-n}q^{n(n+1)/2}}{1+q^n/a}\Big )\notag
\end{align}
and
\begin{equation}
(a;q)_{-n}=\frac{(-1)^na^{-n}q^{n(n+1)/2}}{(q/a;q)_n}.\label{equation:minus-n}
\end{equation}
Equation (\ref{equation:3.4.7-ABII}) follows from a ${}_2\psi_2$ transformation of Bailey \cite{Ba}, see \cite[Entry $3.4.7$]{ABII}.
\end{proof}

\begin{proof}[Proof of Proposition \ref{proposition:eulerian-mxqz-prop}]We prove (\ref{equation:RLNid1}).  Entry $12.3.9$ \cite {ABI}, \cite[p. 5]{RLN}, states that
\begin{equation}
(1+a^{-1})\sum_{n=0}^{\infty}\frac{q^{n+1}(-q)_{2n}}{(aq,q/a;q^2)_{n+1}}=\frac{1}{J_{1,2}}\sum_{n=1}^{\infty}(-1)^{n-1}\Big ( \frac{q^{n^2}}{1-aq^{2n-1}}+\frac{q^{n^2}}{a-q^{2n-1}}\Big ).\label{equation:eq6.3}
\end{equation}
Setting $n=1-r$, we see that
\begin{equation}
\sum_{n=1}^{\infty}\frac{(-1)^{n-1}q^{n^2}}{a-q^{2n-1}}=\sum_{r\le 0}\frac{(-1)^{r}q^{(1-r)^2}}{a-q^{1-2r}}=\sum_{r\le 0}\frac{(-1)^{r-1}q^{r^2}}{1-aq^{2r-1}},
\end{equation}
so the right side of (\ref{equation:eq6.3}) equals
\begin{equation}
-\frac{1}{J_{1,2}}\sum_{n=-\infty}^{\infty}\frac{(-1)^rq^{n^2}}{1-q^{2n-1}a}=-m(a,q^2,q).
\end{equation}
Replacing $a$ by $x$ gives (\ref{equation:RLNid1}).

We prove (\ref{equation:RLNid2}).   Entry $12.4.2$ \cite {ABI}, \cite[p. 5]{RLN}, states that
{\allowdisplaybreaks \begin{align}
j(-a;q)\sum_{n=0}^{\infty}& \frac{(-1)^nq^{n^2}(q;q^2)_n}{(-a;q^2)_{n+1}(-q^2/a;q^2)_n}=1+\sum_{n=1}^{\infty}(a^n+a^{-n}+2(-1)^n)\frac{q^{n(n+1)/2}}{1+q^n}\label{equation:eq6.6}\\
&=\sum_{n=-\infty}^{\infty}\frac{a^nq^{\binom{n+1}{2}}}{1+q^n}+\frac{1}{2}\frac{(q)_{\infty}^2}{(-q)_{\infty}^2}=\sum_{n=-\infty}^{\infty}\frac{q^{\binom{n+1}{2}}a^n}{1+q^n}+\frac{1}{2}J_{1,2}^2.\notag
\end{align}}%
Hence, dividing the extreme left and right of (\ref{equation:eq6.6}) by $j(-a;q)$ gives
\begin{align}
\sum_{n=0}^{\infty}&\frac{(-1)^nq^{n^2}(q;q^2)_n}{(-a;q^2)_{n+1}(-q^2/a;q^2)_n}=\frac{1}{j(-a;q)}\sum_{n=-\infty}^{\infty}\frac{q^{\binom{n+1}{2}}a^n}{1+q^n}+\frac{J_{1,2}^2}{2j(-a;q)}\label{equation:eq6.9}\\
&=a^{-1}m(a^{-1},q,-a)+\frac{J_{1,2}^2}{2j(-a;q)}\notag\\
&=m(a,q,-1)+\frac{J_{1,2}^2}{2j(-a;q)},\notag
\end{align}
by (\ref{equation:alt-mdef-eq}), (\ref{equation:m-fnq-flip}), and (\ref{equation:m-fnq-zflip}).  Changing $a$ to $x$ gives the first equation in (\ref{equation:RLNid2}).  For the second equation use (\ref{equation:m-change-z}).  For the third equation  we recall Entry $12.2.1$ \cite {ABI}, \cite[p. 1] {RLN}:
\begin{equation}
\frac{1}{\overline{J}_{1,4}}\sum_{n=-\infty}^{\infty}\frac{q^{n(n+1)}}{1+q^{2n}a}=\sum_{n=0}^{\infty}\frac{(-1)^nq^{n^2}(q;q^2)_n}{(-a;q^2)_{n+1}(-q^2/a;q^2)_n}.
\end{equation}
Replacing $a$ by $x$ gives
\begin{align*}
\sum_{n=0}^{\infty}&\frac{(-1)^nq^{n^2}(q;q^2)_n}{(-x;q^2)_{n+1}(-q^2/x;q^2)_n}=\frac{1}{\overline{J}_{1,4}}\sum_{n=-\infty}^{\infty} \frac{q^{n(2n+1)}}{1+q^{2n}x}\\
&=\frac{1}{\overline{J}_{1,4}}\sum_{n=-\infty}^{\infty} \Big (\frac{q^{n(2n+1)}}{1-q^{4n}x^2} -x\cdot\frac{q^{n(2n+3)}}{1-q^{4n}x^2} \Big )\\
&=m(-qx^2,q^4,-q^{-1})-q^{-1}xm(-q^{-1}x^2,q^4,-q).&({\text{by } (\ref{equation:alt-mdef-eq})})
\end{align*}

We prove (\ref{equation:RLNid3}).  Entry $12.3.3$ \cite{ABI}, \cite[p. 4]{RLN}, states
\begin{equation*}
\sum_{n=0}^{\infty}\frac{(-1)^nq^{n^2}(q;q^2)_n}{(-q^2a,-q^2/a;q^2)_n}=\sum_{n=0}^{\infty}{}^*\frac{(-1)^n(q;q^2)_n}{(-aq,-q/a)_n}+\frac{1}{2}\frac{J_{1,2}(q;q^2)_{\infty}}{(-qa,-q/a)_{\infty}}.
\end{equation*}
Dividing by $1+a$ and replacing $a$ by $x$, and rewriting the last term, we obtain
\begin{equation*}
\sum_{n=0}^{\infty}\frac{(-1)^nq^{n^2}(q;q^2)_n}{(-x;q^2)_{n+1}(-q^2/x;q^2)_n}=\sum_{n=0}^{\infty}{}^*\frac{(-1)^n(q;q^2)_n}{(-x)_{n+1}(-q/x)_n}+\frac{1}{2}\frac{J_{1,2}^2}{j(-x;q)}.
\end{equation*}

We prove (\ref{equation:RLNid4}).  We recall a (slighty-rewritten) equation from page $8$ of \cite{RLN}, see \cite[Entry $12.3.2$]{ABI}:
\begin{equation*}
\sum_{n=0}^{\infty}\frac{(-1)^n(q;q^2)_nq^{n^2}}{(-a;q^2)_{n+1}(-q^2/a;q^2)_n}-\Big ( 1+\frac{1}{a}\Big )\sum_{n=0}^{\infty}\frac{(-1)^n(q;q^2)_nq^{(n+1)^2}}{(-aq,-q/a;q^2)_{n+1}}=\frac{J_{1,2}^2}{j(-a;q)}.
\end{equation*}
Replace $a$ with $x$, and the result follows from (\ref{equation:RLNid2})

We prove (\ref{equation:RLNid5}).  We recall a (slighty-rewritten) equation from page $5$ of \cite{RLN}, see \cite[Entry $12.4.3$]{ABI}:
\begin{equation*}
\sum_{n=0}^{\infty}\frac{(-1)^n(q^2;q^4)_nq^{2n^2}}{(-a;q^4)_{n+1}(-q^4/a;q^4)_n}+\Big ( 1+\frac{1}{a}\Big )\sum_{n=0}^{\infty}\frac{(-q)_{2n}q^{n+1}}{(aq,q/a;q^2)_{n+1}}=\frac{\overline{J}_{1,4}^2j(-aq^2;q^4)}{j(-a;q^4)j(aq;q^2)}.
\end{equation*}
Replace $a$ with $x$, and the result follows from (\ref{equation:RLNid1}).
\end{proof}


\section{Duals of mock theta functions}\label{section:mock}
We recall two of the fifth order mock theta functions as found in \cite[Section $4$]{HM}:
\begin{align}
\chi_0(q)&=\sum_{n= 0}^{\infty}\frac{q^n}{(q^{n+1})_n}=1+\sum_{n=0}^{\infty}\frac{q^{2n+1}}{(q^{n+1})_{n+1}}=2+3qg(q,q^5)-\frac{J_5^2J_{2,5}}{J_{1,5}^2}\label{equation:mock-chi0-5th}\\
&=2-2m(q^7,q^{15},q^{12})-m(q^7,q^{15},q^9)-2q^{-1}m(q^2,q^{15},q^{12})-q^{-1}m(q^2,q^{15},q^9)\notag\\
\chi_1(q)&=\sum_{n= 0}^{\infty}\frac{q^n}{(q^{n+1})_{n+1}}=1+\sum_{n= 0}^{\infty}\frac{q^{2n+1}(1+q^n)}{(q^{n+1})_{n+1}}=3qg(q^2,q^5)+\frac{J_5^2J_{1,5}}{J_{2,5}^2}\label{equation:mock-chi1-5th}\\
&=-2q^{-1}m(q^4,q^{15},q^{-6})-q^{-1}m(q^4,q^{15},q^3)-2q^{-2}m(q,q^{15},q^{6})-q^{-2}m(q,q^{15},q^{-3})\notag
\end{align}
For the fifth order mock theta function $\chi_0(q)$ we have a dual for each Eulerian form:
\begin{align}
\sum_{n= 0}^{\infty}&\frac{q^n}{(q^{n+1})_n}
\rightarrow \sum_{n= 0}^{\infty}\frac{(-1)^nq^{3n^2/2-n/2}}{(q^{n+1})_n}=2-\sum_{n=0}^{\infty}(-1)^nq^{-7n}q^{15\binom{n+1}{2}}\notag\\
&-q^7\sum_{n=0}^{\infty}(-1)^nq^{7n}q^{15\binom{n+1}{2}} -q\sum_{n=0}^{\infty}(-1)^nq^{-2n}q^{15\binom{n+1}{2}}-q^3\sum_{n=0}^{\infty}(-1)^nq^{2n}q^{15\binom{n+1}{2}}\label{equation:mock-chi0-5th-dualA}\\
1+&\sum_{n= 0}^{\infty}\frac{q^{2n+1}}{(q^{n+1})_{n+1}}
\rightarrow 1+\sum_{n= 0}^{\infty}\frac{(-1)^{n+1}q^{3n^2/2+n/2}}{(q^{n+1})_{n+1}}=1-\sum_{n=0}^{\infty}(-1)^nq^{-7n}q^{15\binom{n+1}{2}}\notag\\
&-q^7\sum_{n=0}^{\infty}(-1)^nq^{7n}q^{15\binom{n+1}{2}} -q\sum_{n=0}^{\infty}(-1)^nq^{-2n}q^{15\binom{n+1}{2}}-q^3\sum_{n=0}^{\infty}(-1)^nq^{2n}q^{15\binom{n+1}{2}}.\label{equation:mock-chi0-5th-dualB}
\end{align}
For the fifth order mock theta function $\chi_1(q)$ we have a dual for each Eulerian form:
{\allowdisplaybreaks {\begin{align}
&\sum_{n= 0}^{\infty}\frac{q^n}{(q^{n+1})_{n+1}}
\rightarrow \sum_{n= 0}^{\infty}\frac{(-1)^{n+1}q^{3\binom{n+1}{2}+1}}{(q^{n+1})_{n+1}}=-q\sum_{n=0}^{\infty}(-1)^nq^{-4n}q^{15\binom{n+1}{2}}\notag\\
&-q^5\sum_{n=0}^{\infty}(-1)^nq^{4n}q^{15\binom{n+1}{2}} -q^2\sum_{n=0}^{\infty}(-1)^nq^{-n}q^{15\binom{n+1}{2}}-q^3\sum_{n=0}^{\infty}(-1)^nq^{n}q^{15\binom{n+1}{2}}\label{equation:mock-chi1-5th-dualA}\\
&1+\sum_{n= 0}^{\infty}\frac{q^{2n+1}(1+q^n)}{(q^{n+1})_{n+1}}
\rightarrow \notag\\
&1+\sum_{n= 0}^{\infty}\frac{(-1)^{n+1}q^{3n^2/2-n/2}(1+q^n)}{(q^{n+1})_{n+1}}=-1-q\sum_{n=0}^{\infty}(-1)^nq^{-4n}q^{15\binom{n+1}{2}}\label{equation:mock-chi1-5th-dualB} \\
& \ \ \ \ \ -q^5\sum_{n=0}^{\infty}(-1)^nq^{4n}q^{15\binom{n+1}{2}}
-q^2\sum_{n=0}^{\infty}(-1)^nq^{-n}q^{15\binom{n+1}{2}}-q^3\sum_{n=0}^{\infty}(-1)^nq^{n}q^{15\binom{n+1}{2}}.\notag
\end{align}}%

\begin{theorem}\label{theorem:fifth-duals}  Identities (\ref{equation:mock-chi0-5th-dualA})-(\ref{equation:mock-chi1-5th-dualA}) are true.
\end{theorem}
\begin{proof}[Proof of Theorem \ref{theorem:fifth-duals}]  For all of the identities we will employ Lemma \ref{lemma:important-lemma}.  For identity (\ref{equation:mock-chi0-5th-dualA}) we use the Bailey pair \cite[p. $12$]{Wa1}:
\begin{align*}
\alpha_n&=(-1)^{\lfloor (4n+1)/3 \rfloor}q^{(n-2)n/3}\frac{1-q^{2n+1}}{1-q}\chi(n\not\equiv 1 \pmod 3),\ \ \ \beta_n=\frac{q^{n(n-1)}}{(q;q)_{2n}},
\end{align*}
where $\chi(\text{true})=1$ and $\chi(\text{false})=0$.  For identity (\ref{equation:mock-chi0-5th-dualB}) we use the Bailey pair \cite[A6]{S}:
\begin{align*}
\alpha_{3n-1}=q^{3n^2+n}, \ \alpha_{3n}&=q^{3n^2-n}, \ \alpha_{3n+1}=-q^{3n^2+n}-q^{3n^2+5n+2},
\ \ \ \beta_n=\frac{q^{n^2}}{(q^2;q)_{2n}}.
\end{align*}
 For identity (\ref{equation:mock-chi1-5th-dualA}) we use the Bailey pair \cite[A8]{S}:
\begin{equation*}
\alpha_{3n-1}=q^{3n^2-2n}, \ \alpha_{3n}=q^{3n^2+2n}, \ \alpha_{3n+1}=-q^{3n^2+4n+1}-q^{3n^2+2n},\ \ \ 
\beta_n=\frac{q^{n^2+n}}{(q^2;q)_{2n}}.\qedhere
\end{equation*}
\end{proof}
We recall the three seventh order mock theta functions as found in \cite[Section $4$]{HM}:
 \begin{align}
{\mathcal{F}}_0(q)&=\sum_{n= 0}^{\infty}\frac{q^{n^2}}{(q^{n+1})_{n}}=2+2qg(q,q^{7})-\frac{J_{3,7}^2}{J_1}\label{equation:mock-F0-7th}\\
&=m(q^{10},q^{21},q^9)+m(q^{10},q^{21},q^{-9})-q^{-1}m(q^{4},q^{21},q^9)-q^{-1}m(q^{4},q^{21},q^{-9}) \notag\\
{\mathcal{F}}_1(q)&=\sum_{n= 1}^{\infty}\frac{q^{n^2}}{(q^{n})_{n}}=2q^2g(q^2,q^{7})+\frac{qJ_{1,7}^2}{J_1}\label{equation:mock-F1-7th}\\
&=-m(q^{8},q^{21},q^3)-m(q^{8},q^{21},q^{-3})-q^{-2}m(q,q^{21},q^3)-q^{-2}m(q,q^{21},q^{-3}) \notag\\
{\mathcal{F}}_2(q)&=\sum_{n= 0}^{\infty}\frac{q^{n(n+1)}}{(q^{n+1})_{n+1}}=2q^2g(q^3,q^{7})+\frac{J_{2,7}^2}{J_1}\label{mock-F2-7th}\\
&=-q^{-1}m(q^{5},q^{21},q^6)-q^{-1}m(q^{5},q^{21},q^{-6})-q^{-2}m(q^2,q^{21},q^6)-q^{-2}m(q^2,q^{21},q^{-6}) \notag
\end{align}
The substitution $q\rightarrow q^{-1}$,  the heuristic (iii), and identity (\ref{equation:m-fnq-flip}), lead us to the duals:
{\allowdisplaybreaks \begin{align}
{\mathcal{F}}_0(q)&\rightarrow \sum_{n= 0}^{\infty}\frac{(-1)^nq^{\binom{n+1}{2}}}{(q^{n+1})_{n}}=\sum_{n=0}^{\infty}(-1)^nq^{-10n}q^{21\binom{n+1}{2}}+q^{10}\sum_{n=0}^{\infty}(-1)^nq^{10n}q^{21\binom{n+1}{2}}\notag\\
&\ \ \ \ \ -q\sum_{n=0}^{\infty}(-1)^nq^{-4n}q^{21\binom{n+1}{2}}-q^5\sum_{n=0}^{\infty}(-1)^nq^{4n}q^{21\binom{n+1}{2}},\label{equation:mock-F0-7th-dual}\\
{\mathcal{F}}_1(q)&\rightarrow \sum_{n= 0}^{\infty}\frac{(-1)^{n+1}q^{\binom{n+1}{2}}}{(q^{n+1})_{n+1}}=-\sum_{n=0}^{\infty}(-1)^nq^{-8n}q^{21\binom{n+1}{2}}-q^{8}\sum_{n=0}^{\infty}(-1)^nq^{8n}q^{21\binom{n+1}{2}}\notag\\
&\ \ \ \ \ -q^2\sum_{n=0}^{\infty}(-1)^nq^{-n}q^{21\binom{n+1}{2}}-q^3\sum_{n=0}^{\infty}(-1)^nq^{n}q^{21\binom{n+1}{2}},\label{equation:mock-F1-7th-dual}\\
{\mathcal{F}}_2(q)&\rightarrow \sum_{n= 0}^{\infty}\frac{(-1)^{n+1}q^{\binom{n+2}{2}}}{(q^{n+1})_{n+1}}=-q\sum_{n=0}^{\infty}(-1)^nq^{-5n}q^{21\binom{n+1}{2}}-q^{6}\sum_{n=0}^{\infty}(-1)^nq^{5n}q^{21\binom{n+1}{2}}\notag\\
&\ \ \ \ \ -q^2\sum_{n=0}^{\infty}(-1)^nq^{-2n}q^{21\binom{n+1}{2}}-q^4\sum_{n=0}^{\infty}(-1)^nq^{2n}q^{21\binom{n+1}{2}}.\label{equation:mock-F2-7th-dual}
\end{align}}

\begin{theorem}\label{theorem:seventh-duals}  Identities (\ref{equation:mock-F0-7th-dual})-(\ref{equation:mock-F2-7th-dual}) are true.
\end{theorem}
\begin{proof}[Proof of Theorem \ref{theorem:seventh-duals}]  For all of the identities we will employ Lemma \ref{lemma:important-lemma}.  For identity (\ref{equation:mock-F0-7th-dual}) we use the Bailey pair \cite[$(4.6)$]{Wa1}:
\begin{align*}
\alpha_n&=(-1)^{\lfloor (4n+1)/3 \rfloor}q^{(2n-1)n/3}\frac{1-q^{2n+1}}{1-q}\chi(n\not\equiv 1 \pmod 3),\ \ \ 
\beta_n=\frac{1}{(q;q)_{2n}},
\end{align*}
where $\chi(\text{true})=1$ and $\chi(\text{false})=0$.  For identity (\ref{equation:mock-F1-7th-dual}) we use the Bailey pair \cite[A2]{S}:
\begin{align*}
\alpha_{3n-1}=q^{6n^2-n}, \ \alpha_{3n}&=q^{6n^2+n}, \ \alpha_{3n+1}=-q^{6n^2+5n+1}-q^{6n^2+7n+2},\ \ \ 
\beta_n=\frac{1}{(q^2;q)_{2n}}.
\end{align*}
 For identity (\ref{equation:mock-F2-7th-dual}) we use the Bailey pair   \cite[A4]{S}:
\begin{align*}
\alpha_{3n-1}=q^{6n^2-4n}, \ \alpha_{3n}&=q^{6n^2+4n}, \ \alpha_{3n+1}=-q^{6n^2+8n+2}-q^{6n^2+4n},\ \ \ 
\beta_n=\frac{q^{n}}{(q^2;q)_{2n}}.\qedhere
\end{align*}
\end{proof}

We recall the four tenth order mock theta functions as found in  \cite[Section $4$]{HM}:
{\allowdisplaybreaks \begin{align}
{\phi}(q)&=\sum_{n= 0}^{\infty}\frac{q^{\binom{n+1}{2}}}{(q;q^2)_{n+1}}=-q^{-1}m(q,q^{10},q)-q^{-1}m(q,q^{10},q^{2})\label{equation:mock-phi-10th}\\
{\psi}(q)&=\sum_{n= 0}^{\infty}\frac{q^{\binom{n+2}{2}}}{(q;q^2)_{n+1}}=-m(q^3,q^{10},q)-m(q^3,q^{10},q^{3})\label{equation:mock-psi-10th}\\
{X}(q)&=\sum_{n= 0}^{\infty}\frac{(-1)^nq^{n^2}}{(-q)_{2n}}
=m(-q^2,q^{5},q)+m(-q^2,q^{5},q^{4})\label{equation:mock-X-10th}\\
{\chi}(q)&=\sum_{n=0}^{\infty}\frac{(-1)^nq^{(n+1)^2}}{(-q)_{2n+1}}
=m(-q,q^{5},q^2)+m(-q,q^{5},q^{3})\label{equation:mock-chi-10th}
\end{align}}
The substitution $q\rightarrow q^{-1}$,  the heuristic (iii), and identity (\ref{equation:m-fnq-flip}), lead us to the duals:
\begin{align}
\phi(q)&\rightarrow \sum_{n= 0}^{\infty}\frac{(-1)^{n+1}q^{\binom{n+2}{2}}}{(q;q^2)_{n+1}}=-q\sum_{n=0}^{\infty}(-1)^nq^{-n}q^{5n(n+1)}-q^{2}\sum_{n=0}^{\infty}(-1)^nq^{n}q^{5n(n+1)},\label{equation:mock-phi-10th-dual}\\
\psi(q)&\rightarrow \sum_{n= 0}^{\infty}\frac{(-1)^{n+1}q^{\binom{n+1}{2}}}{(q;q^2)_{n+1}}=-\sum_{n=0}^{\infty}(-1)^nq^{-3n}q^{5n(n+1)}-q^{3}\sum_{n=0}^{\infty}(-1)^nq^{3n}q^{5n(n+1)},\label{equation:mock-psi-10th-dual}\\
X(q)&\rightarrow \sum_{n= 0}^{\infty}\frac{(-1)^nq^{n(n+1)}}{(-q)_{2n}}=\sum_{n=0}^{\infty}q^{-2n}q^{5n(n+1)/2}-q^2\sum_{n=0}^{\infty}q^{2n}q^{5n(n+1)/2},\label{equation:mock-X-10th-dual}\\
\chi(q)&\rightarrow \sum_{n= 0}^{\infty}\frac{(-1)^{n}q^{n(n+1)}}{(-q)_{2n+1}}=\sum_{n=0}^{\infty}q^{-n}q^{5n(n+1)/2}-q\sum_{n=0}^{\infty}q^{n}q^{5n(n+1)/2}.\label{equation:mock-chi-10th-dual}
\end{align}

\begin{theorem}\label{theorem:tenth-duals}  Identities (\ref{equation:mock-phi-10th-dual})-(\ref{equation:mock-chi-10th-dual}) are true.
\end{theorem}
\noindent The following is then easy to show:
\begin{corollary} Let $\omega$ be a primitive third root of unity and denote the duals of the tenth order mock theta functions by $\phi_D(q)$, $\psi_D(q)$, $X_D(q)$, and $\chi_D(q)$ respectively.  Then
\begin{align}
q^{-2/3}\phi_D(q^3)-\frac{\psi_D(\omega^2 q^{1/3})-\psi_D(\omega q^{1/3})}{\omega - \omega^2}&=0,\label{equation:tenth-dual-I}\\
q^{2/3}\psi_D(q^3)+\frac{\omega \phi_D(\omega^2 q^{1/3})-\omega^2\phi_D(\omega q^{1/3})}{\omega - \omega^2}&=0,\label{equation:tenth-dual-II}\\
X_D(q^3)-\frac{\omega \chi_D(\omega^2 q^{1/3})-\omega^2\chi_D(\omega q^{1/3})}{\omega - \omega^2}&=0,\label{equation:tenth-dual-III}\\
\chi_D(q^3)+q^{-2/3}\frac{ X_D(\omega^2 q^{1/3})-X_D(\omega q^{1/3})}{\omega - \omega^2}&=0.\label{equation:tenth-dual-IV}
\end{align}
\end{corollary}
\noindent For comparison, we recall the four identities for the tenth order mock theta functions \cite{CH1,CH2,Z10}.  Again, $\omega$ is a primitive third root of unity:
{\allowdisplaybreaks \begin{align*}
q^{2/3}\phi(q^3)-\frac{\psi(\omega q^{1/3})-\psi(\omega^2 q^{1/3})}{\omega - \omega^2}&=-q^{1/3}\frac{\sum_{n\in \mathbb{Z}} (-1)^nq^{n^2/3}}{\sum_{n\in \mathbb{Z}} (-1)^nq^{n^2}}\frac{\sum_{n\in \mathbb{Z}} (-1)^nq^{5n^2/2+3n/2}}{(q;q^2)_{\infty}},\\
q^{-2/3}\psi(q^3)+\frac{\omega \phi(\omega q^{1/3})-\omega^2\phi(\omega^2 q^{1/3})}{\omega - \omega^2}&=\frac{\sum_{n\in \mathbb{Z}} (-1)^nq^{n^2/3}}{\sum_{n\in \mathbb{Z}} (-1)^nq^{n^2}}\frac{\sum_{n\in \mathbb{Z}} (-1)^nq^{5n^2/2+n/2}}{(q;q^2)_{\infty}},\\
X(q^3)-\frac{\omega \chi(\omega q^{1/3})-\omega^2\chi(\omega^2 q^{1/3})}{\omega - \omega^2}&=\frac{\sum_{n\in \mathbb{Z}} (-1)^nq^{n(n+1)/6}}{\sum_{n\in \mathbb{Z}} (-1)^nq^{n(n+1)/2}}\frac{\sum_{n\in \mathbb{Z}} (-1)^nq^{5n^2+n}}{(-q;q)_{\infty}},\\
\chi(q^3)+q^{2/3}\frac{ X(\omega q^{1/3})-X(\omega^2 q^{1/3})}{\omega - \omega^2}&=-\frac{\sum_{n\in \mathbb{Z}} (-1)^nq^{n(n+1)/6}}{\sum_{n\in \mathbb{Z}} (-1)^nq^{n(n+1)/2}}\frac{\sum_{n\in \mathbb{Z}} (-1)^nq^{5n^2+3n}}{(-q;q)_{\infty}}.
\end{align*}}%

\begin{proof}[Proof of Theorem \ref{theorem:tenth-duals}]  For all of the identities we will employ Lemma \ref{lemma:important-lemma}.   For identity (\ref{equation:mock-phi-10th-dual}) we use the Bailey pair \cite[C4]{S}:
\begin{align*}
\alpha_{2n}&=(-1)^nq^{3n^2+3n}, \ \alpha_{2n+1}=(-1)^{n+1}q^{3n^2+3n}, \ \ \ \beta_n=\frac{q^n}{(q^3;q^2)_n(q;q)_{n}}.
\end{align*}
For identity (\ref{equation:mock-psi-10th-dual}) we use the Bailey pair \cite[C3]{S}:
\begin{align*}
\alpha_{2n}&=(-1)^nq^{3n^2+n}, \ \alpha_{2n+1}=(-1)^{n+1}q^{3n^2+5n+2}, \ \ \ 
\beta_n=\frac{1}{(q^3;q^2)_n(q;q)_{n}}.
\end{align*}
For identity (\ref{equation:mock-X-10th-dual}) we use the Bailey pair \cite[$(4.4)$]{Wa1}:
\begin{align*}
\alpha_{n}&=\frac{(-1)^nq^{(3n-1)n/4}(1-q^{2n+1})}{(1-q)},\ \ \ \beta_n=\frac{1}{(q^2;q^2)_n(-q^{1/2};q)_{n}}.
\end{align*}
 For identity (\ref{equation:mock-chi-10th-dual}) we use the Bailey pair   \cite[G2]{S}:
\begin{align*}
\alpha_{2n}&=q^{3n^2+\tfrac{1}{2}n}\frac{(1-q^{2n+\tfrac{1}{2}})}{1-q^{\tfrac{1}{2}}},\ \alpha_{2n-1}=q^{3n^2-\tfrac{1}{2}n}\frac{(1-q^{-2n+\tfrac{1}{2}})}{1-q^{\tfrac{1}{2}}}, \ \ \ 
\beta_n=\frac{1}{(q^2;q^2)_n(-q^{3/2};q)_{n}}.\qedhere
\end{align*}
\end{proof}


\section{Duals and Duals of second type}\label{section:duals}

In this section we state the duals for many mixed partial theta functions found in \cite{RLN} and find the corresponding duals of second type.  The duals are not new, but what we do is to rewrite them in terms of Appell--Lerch sums.  With the Appell--Lerch sum form in mind, we look for the functional equations of the duals of second type by assuming a cancellation similar to that which occurs for the two sixth order mock theta functions $\phi(q)$ and $\phi \_(q)$ in identity \cite[Entry $3.4.1$]{ABII}, \cite[p. 6, 14]{RLN} and for the two sixth orders $\psi(q)$ and $\psi \_(q)$ in  identity \cite[Entry $3.4.2$]{ABII}, \cite[p. 14]{RLN}.

Our first mixed partial theta function (\ref{equation:ABII-6.3.2}) is well-known and was published by Andrews \cite{A1} immediately after his discovery of the lost notebook.  For the interested reader, we point out that (\ref{equation:ABII-6.3.2}) yields as special cases the two partial theta function identities found in Lawrence and Zagier's work on quantum invariants of $3$-manifolds \cite{LZ}.

\subsection{Entry $6.3.2$ \cite{ABII}, also \cite[p. 7]{RLN}}

For $a\ne 0$,
{\allowdisplaybreaks \begin{align}
\sum_{n=0}^{\infty}\frac{q^n}{(-aq,-q/a)_n}&=(1+a)\sum_{n=0}^{\infty}a^{3n}q^{n(3n+1)/2}(1-a^2q^{2n+1})\label{equation:ABII-6.3.2}\\
& \ \ \ \ \ \ \ \ \ \ -\frac{(1+a)J_1}{j(-a;q)}\sum_{n=0}^{\infty}(-1)^na^{2n+1}q^{n(n+1)/2}.\notag
\end{align}}%

The dual already exists.  We note from identity (\ref{equation:g-def}) and (\ref{equation:g-to-m}) that
\begin{align}
\sum_{n=0}^{\infty}\frac{q^{n^2}}{(-aq,-q/a)_n}&=(1+a)(1-ag(-a,q))\label{equation:6.3.2-dual}\\
&=(1+a)(1-m(-q^2a^{-3},q^3,a^2)+a^{-1}m(-qa^{-3},q^3,a^2)).\notag
\end{align}
We demonstrate how the heuristic (iii) takes us from the partial theta function on the right-hand side of (\ref{equation:ABII-6.3.2}) to the right-hand side of (\ref{equation:6.3.2-dual}):
\begin{align*}
\sum_{n=0}^{\infty}&(a^3q^{-1})^nq^{3\binom{n+1}{2}}-a^2q\sum_{n=0}^{\infty}(a^3q)^nq^{3\binom{n+1}{2}} \\
&\rightarrow \sum_{n=0}^{\infty}(a^3q)^nq^{-3\binom{n+1}{2}}-a^2q^{-1}\sum_{n=0}^{\infty}(a^3q^{-1})^nq^{-3\binom{n+1}{2}} &(q\rightarrow q^{-1})\\
& \ \ \sim m(-a^3q,q^3,*)-a^2q^{-1}m(-a^3q^{-1},q^3,*)&(\text{by }(iii))\\
& \ \ \sim  1+a^3q^{-2}m(-a^3q^{-2},q^3,*)-a^2q^{-1}m(-a^3q^{-1},q^3,*)&({\text{by }} (\ref{equation:m-fnq-x}))\\
& \ \ \sim  1-m(-q^2a^{-3},q^3,*)+a^{-1}m(-qa^{-3},q^3,*).&({\text{by }} (\ref{equation:m-fnq-flip}))
\end{align*}
We consider the sum
\begin{equation}
\sum_{n=-\infty}^{-1}\frac{q^{n^2}}{(-aq,-q/a)_n}.\label{equation:6.3.2-ABII-tail}
\end{equation}
Making the substution $n\rightarrow -n$, (\ref{equation:6.3.2-ABII-tail}) becomes
\begin{equation}
\sum_{n=1}^{\infty}q^{n}(-1/a,-a)_n=\Big ( 1+\frac{1}{a}\Big )\Big ( 1+a \Big )\sum_{n=0}^{\infty}q^{n+1}(-q/a,-aq)_n=:(1+a)f(a)\label{equation:6.3.2-ABII-tail-2}
\end{equation}
We find
\begin{equation}
f(qa)+1-qa^2-qa^3f(a)=(1-a^2q)\frac{j(-a;q)}{J_1}\label{equation:dog}
\end{equation}
We rewrite the functional equation (\ref{equation:dog}) as
\begin{equation}
f(a)=q^{-1}a^{-3}-a^{-1}-q^{-1}a^{-3}(1-a^2q)\frac{j(-a;q)}{J_1}+q^{-1}a^{-3}f(qa).\label{equation:cat}
\end{equation}
Iterating the functional equation (\ref{equation:cat}) and using the heuristic (iii) suggests the identity
\begin{equation*}
f(a)\sim-1+ag(-a,q)+\frac{j(-a;q)}{J_1}m(a^2,q,*).
\end{equation*}
Some numerical work suggests the dual of second type:
\begin{align}
\Big ( 1+\frac{1}{a}\Big )&\sum_{n=0}^{\infty}q^{n+1}(-q/a,-aq)_n=-1+ag(-a,q)+\frac{j(-a;q)}{J_1}m(a^2,q,-a^{-1}).\label{equation:6.3.2-ABII-2ndDualB}\\
&=-1+ag(-a,q)+\frac{j(-a;q)}{J_1}m(a^2,q,-1) +\frac{1}{2}\frac{j(a;q)^3j(qa^2;q^2)}{J_2^2j(a^4;q^2)}\label{equation:6.3.2-ABII-2ndDualA}
\end{align}
\begin{theorem} Identities (\ref{equation:6.3.2-ABII-2ndDualB}) and  (\ref{equation:6.3.2-ABII-2ndDualA}) are true
\end{theorem}
\begin{proof}
Using \cite[Entry $3.4.7$]{ABII} (\ref{equation:3.4.7-ABII}) with $b=1/a$, we find that
\begin{align*}
\frac{1}{1+a}\sum_{n=0}^{\infty}\frac{q^{n^2}}{(-aq,-q/a)_n}+\Big ( 1+\frac{1}{a}\Big )\sum_{n=0}^{\infty}q^{n+1}(-q/a,-aq)_n=\frac{j(-a;q)}{J_1}m(a^2,q,-a^{-1}).
\end{align*}
Identity (\ref{equation:6.3.2-ABII-2ndDualB}) follows by (\ref{equation:6.3.2-dual}).  Identity (\ref{equation:6.3.2-ABII-2ndDualA}) then follows from (\ref{equation:6.3.2-ABII-2ndDualB}) by using (\ref{equation:m-change-z}).
\end{proof}
\begin{remark} 
In \cite{L}, Lovejoy showed (slightly rewritten) via Bailey pairs that
\begin{equation}
1+\Big ( 1+\frac{1}{a}\Big )\sum_{n=0}^{\infty}q^{n+1}(-q/a,-aq)_n=aq^3f_{3,2,1}(q^6,-aq^3,q)/J_{1}.
\end{equation}
\end{remark}

\subsection{Entry $6.3.4$ \cite{ABII}, also \cite[p. 37]{RLN}}
If $a\ne 0$, then
\begin{align}
\sum_{n=0}^{\infty}\frac{q^{2n+1}}{(-aq,-q/a;q^2)_{n+1}}
&=\sum_{n=0}^{\infty}a^{3n+1}q^{3n^2+2n}(1-aq^{2n+1})\label{equation:ABII-6.3.4}\\
& \ \ \ \ \ \ \ \ \ \ -\frac{J_2}{j(-aq;q^2)}\sum_{n=0}^{\infty}(-1)^na^{2n+1}q^{n(n+1)}.\notag
\end{align}
The dual is just the identity from Proposition \ref{proposition:newgid}, with $q\rightarrow q^2$ and $x\rightarrow -aq$, i.e.,
\begin{equation}
\sum_{n=0}^{\infty}\frac{q^{2n^2+2n+1}}{(-aq,-q/a;q^2)_{n+1}}=qg(-aq,q^2).\label{equation:ABII-6.3.4-dual}
\end{equation}
Let us consider 
\begin{equation}
\sum_{n=-\infty}^{-1}\frac{q^{2n^2+2n+1}}{(-aq,-q/a;q^2)_{n+1}}.\label{equation:ABII-6.3.4-tail}
\end{equation}
Making the substitution $n\rightarrow -n$, (\ref{equation:ABII-6.3.4-tail}) becomes
\begin{equation}
\sum_{n=1}^{\infty}q^{2n-1}(-aq,-q/a;q^2)_{n-1}=\sum_{n=0}^{\infty}q^{2n+1}(-aq,-q/a;q^2)_{n}=:f(a).\label{equation:ABII-6.3.4-tail-2}
\end{equation}
Numerically, we find
\begin{equation}
f(q^2a)+aq^2-a^2q^3-a^3q^3f(a)=q\cdot(1-a^2q^2)\frac{j(-aq;q^2)}{J_2}.
\end{equation}
Using the heuristic and some more numerical work suggests
\begin{align}
f(a)&=-qg(-aq,q^2)+a\cdot \frac{j(-aq;q^2)}{J_2}m(a^2,q^2,-1)-\frac{1}{2}\frac{aj(aq;q^2)^3j(a^2;q^4)}{J_4^2j(a^4;q^4)}\label{equation:6.3.4-ABII-2ndDualA}\\
&=-qg(-aq,q^2)+a\cdot \frac{j(-aq;q^2)}{J_2}m(a^2,q^2,-a^{-1}q).\label{equation:6.3.4-ABII-2ndDualB}
\end{align}
\begin{theorem} Identities (\ref{equation:6.3.4-ABII-2ndDualA}) and (\ref{equation:6.3.4-ABII-2ndDualB}) are true
\end{theorem}
\begin{proof}
Using \cite[Entry $3.4.7$]{ABII} (\ref{equation:3.4.7-ABII}) with $q\rightarrow q^2$, $a\rightarrow a/q$, $b\rightarrow 1/aq$, we find that
\begin{align*}
\frac{1+aq}{a}\sum_{n=0}^{\infty}\frac{q^{2n^2+2n+1}}{(-aq,-q/a;q^2)_{n+1}}&+q\Big ( 1+\frac{1}{aq}\Big )\sum_{n=0}^{\infty}q^{2n+1}(-q/a,-aq;q^2)_n\notag\\
&=(1+aq)\frac{j(-aq;q^2)}{J_2}m(a^2,q^2,-a^{-1}q).
\end{align*}
Identity (\ref{equation:6.3.4-ABII-2ndDualB}) follows by (\ref{equation:ABII-6.3.4-dual}).  Identity (\ref{equation:6.3.4-ABII-2ndDualA}) follows from (\ref{equation:6.3.4-ABII-2ndDualB}) by using (\ref{equation:m-change-z}).
\end{proof}
We note that the methods of \cite{L} give
\begin{equation}
\sum_{n=0}^{\infty}q^{2n+1}(-aq,-q/a;q^2)_{n}=qf_{3,2,1}(q^6,-aq^3,q^2)/J_2.
\end{equation}

\subsection{Entry $6.3.6$ \cite{ABII}, also \cite[p. 8]{RLN}}
If $a\ne 0$, then
\begin{align}
\Big ( 1+\frac{1}{a}\Big )\sum_{n=0}^{\infty}\frac{(q;q^2)_nq^{2n+1}}{(-aq,-q/a;q^2)_{n+1}}
&=\sum_{n=0}^{\infty}(-1)^na^{n}q^{n(n+1)/2}\label{equation:ABII-6.3.6}\\
& \ \ \ \ \ \ \ \ \ \ -\frac{J_1}{j(-aq,q^2)}\sum_{n=0}^{\infty}a^{3n}q^{n(3n+1)}(1-a^2q^{4n+2}).\notag
\end{align}
The dual is just identity (\ref{equation:RLNid4}) of Proposition \ref{proposition:eulerian-mxqz-prop}:
\begin{equation}
\Big ( 1+\frac{1}{a}\Big )\sum_{n=0}^{\infty}\frac{(q;q^2)_n(-1)^nq^{(n+1)^2}}{(-aq,-q/a;q^2)_{n+1}}=m(a,q,-1)-\frac{J_{1,2}^2}{2j(-a;q)}.\label{equation:ABII-6.3.6-dual}
\end{equation}
For the dual of second type, iterating the functional equation leads to
\begin{align}
f(a):=\Big ( 1+\frac{1}{a}\Big )&\sum_{n=0}^{\infty}\frac{q^{2n+1}(-aq,-q/a;q^2)_{n}}{(q;q^2)_{n+1}}\label{equation:ABII-6.3.6-dualtypeII}\\
&=-m(a,q,-1)+\frac{j(-aq;q^2)}{J_1}\Big ( 1-ag(-a,q^2)\Big )-\frac{1}{2}\frac{J_{1,2}^2}{j(-a;q)},\notag
\end{align}
where the functional equation is
\begin{equation}
f(q^2a)+1-qa-qa^2f(a)=\frac{j(-aq;q^2)}{J_1}\frac{(1-a^2q^2)}{aq}.
\end{equation}
Using the method of \cite{L} we have
\begin{align*}
\Big ( 1+\frac{1}{a}\Big )&\sum_{n=0}^{\infty}\frac{q^{2n+1}(-aq,-q/a;q^2)_{n}}{(q;q^2)_{n+1}}\\
&=\frac{1}{J_1}\Big ( -q^7f_{3,3,2}(-q^{19},-a^2q^{16},q^4)+qf_{3,3,2}(-q^{11},-a^2q^{8},q^4)\\
&\ \ \ \ \ +aq^{4}f_{3,3,2}(-q^{17},-a^2q^{12},q^4)-aq^{14}f_{3,3,2}(-q^{25},-a^2q^{20},q^4)\Big).
\end{align*}

\subsection{Entry $6.3.7$ \cite{ABII}, also \cite[p. 2]{RLN}}
If $a\ne 0$, then
\begin{align}
\Big ( 1+\frac{1}{a}\Big )&\sum_{n=0}^{\infty}\frac{(-q)_{2n}q^{2n+1}}{(aq,q/a;q^2)_{n+1}}
=-\sum_{n=0}^{\infty}(-a)^{n}q^{n(n+1)} +\frac{\overline{J}_{1,4}}{j(aq;q^2)}\sum_{n=0}^{\infty}(-a)^{n}q^{n(n+1)/2}.\label{equation:ABII-6.3.7}
\end{align}
The dual is just identity (\ref{equation:RLNid1}) of  Proposition \ref{proposition:eulerian-mxqz-prop}:
\begin{equation*}
(1+x^{-1})\sum_{n= 0}^{\infty}\frac{q^{n+1}(-q)_{2n}}{(qx,q/x;q^2)_{n+1}}=-m(x,q^2,q).
\end{equation*}
For the dual of second type, iterating the functional equation leads to
\begin{align}
f(a):=\Big ( 1+\frac{1}{a}\Big )&\sum_{n=0}^{\infty}\frac{(aq,q/a;q^2)_nq^{2n+1}}{(-q)_{2n+1}}\label{equation:ABII-6.3.7-dualtypeII}\\
&=2m(a,q^2,-1)-\frac{j(aq;q^2)}{\overline{J}_{1,4}}m(a,q,-1)
 -\frac{1}{2}\frac{J_1^5}{J_2^4}\frac{j(aq;q^2)}{j(-a;q)},\notag
\end{align}
where the functional equation is
\begin{equation}
f(q^2a)-2+af(a)=\frac{(1-aq)}{aq}\frac{j(aq;q^2)}{\overline{J}_{1,4}}.
\end{equation}
\begin{theorem}
Identity (\ref{equation:ABII-6.3.7-dualtypeII}) is true.
\end{theorem}
\begin{proof} The methods of \cite{L} yield
\begin{equation}
\Big ( 1+\frac{1}{a}\Big )\sum_{n=0}^{\infty}\frac{(aq,q/a;q^2)_nq^{2n+1}}{(-q)_{2n+1}}=qf_{2,2,1}(aq^3,-q^2,q)/\overline{J}_{1,4}.
\end{equation}
The result then follows from Proposition \ref{proposition:f221}.
\end{proof}
Equation ({\ref{equation:ABII-6.3.7-dualtypeII}}) is the dual of  \cite[Entry $5.4.4$]{ABII}, also \cite[p. 15]{RLN}:
For $a\ne0$,
\begin{align}
\Big ( 1+\frac{1}{a}\Big )\sum_{n=0}^{\infty}\frac{(aq,q/a;q^2)_nq^{n}}{(-q)_{2n+1}}
&=\sum_{n=0}^{\infty}(-1)^n (a^n+a^{-n-1})q^{n(n+1)}\label{equation:ABII-5.4.4}.
\end{align}

\subsection{Entry $6.3.9$ \cite{ABII}, also \cite[p. 29]{RLN}}
For $a\ne 0$,
\begin{align}
\sum_{n=0}^{\infty}\frac{(q;q^2)_{n}q^{2n}}{(-aq^2,-q^2/a;q^2)_{n}}
&=(1+a)\sum_{n=0}^{\infty}(-1)^na^{n}q^{n(n+1)/2}\label{equation:ABII-6.3.9}\\
& \ -\frac{a(1+a)J_1}{j(-a;q^2)}\sum_{n=0}^{\infty}a^{3n}q^{3n^2+2n}(1-aq^{2n+1}).\notag
\end{align}
The dual is just identity (\ref{equation:RLNid2}) of  Proposition \ref{proposition:eulerian-mxqz-prop}:
\begin{equation}
\sum_{n=0}^{\infty}\frac{(q;q^2)_{n}(-1)^nq^{n^2}}{(-a;q^2)_{n+1}(-q^2/a;q^2)_{n}}= m(a,q,-1)+\frac{J_{1,2}^2}{2j(-a;q)}.\label{equation:6.3.9-dual-final}
\end{equation}
For the dual of second type, iterating the functional equation leads to
\begin{align}
f(a):=\Big ( 1+\frac{1}{a}\Big)&\sum_{n=0}^{\infty}\frac{q^{2n+2}(-aq^2,-q^2/a;q^2)_{n}}{(q;q^2)_{n+1}}\label{equation:ABII-6.3.9-dualtypeII}\\
&=-m(a,q,-1)+\frac{j(-a;q^2)}{J_1}\frac{q}{a}g(-aq,q^2)+\frac{1}{2}\frac{j(a;q)J_{1,2}}{j(a^2;q^2)},\notag
\end{align}
where the functional equation is
\begin{equation}
f(q^2a)+1-qa-qa^2f(a)=\frac{(1-aq)}{a}\frac{j(-a;q^2)}{J_1}.\label{equation:ABII-6.3.9-func}
\end{equation}
Using the method of \cite{L} we have
\begin{align*}
\Big ( 1+\frac{1}{a}\Big)&\sum_{n=0}^{\infty}\frac{q^{2n+2}(-aq^2,-q^2/a;q^2)_{n}}{(q;q^2)_{n+1}}\\
&=\frac{1}{J_1}\Big ( aq^6f_{3,3,2}(-q^{19},-a^2q^{14},q^4)-q^5f_{3,3,2}(-q^{17},-a^2q^{14},q^4)\\
&\ \ \ \ \ -aq^{11}f_{3,3,2}(-q^{23},-a^2q^{18},q^4)+q^2f_{3,3,2}(-q^{13},-a^2q^{10},q^4)\Big).
\end{align*}

\subsection{Entry $6.3.11$ \cite{ABII}, also \cite[p. 4]{RLN}}
For $a\ne 0$,
\begin{align}
\sum_{n=0}^{\infty}\frac{(q;q^2)_{n}q^{n}}{(-aq,-q/a)_{n}}
&=(1+a)\sum_{n=0}^{\infty}(-1)^na^{n}q^{n(n+1)/2}\label{equation:ABII-6.3.11}\\
& \ \ \ \ \ \ \ \ \ \ -\frac{a(1+a)J_{1,2}}{j(-a;q)}\sum_{n=0}^{\infty}(-1)^na^{2n}q^{n(n+1)}.\notag
\end{align}
The dual is just identity (\ref{equation:RLNid3}) of  Proposition \ref{proposition:eulerian-mxqz-prop}:
\begin{equation*}
\sum_{n= 0}^{\infty}{}^*\frac{(-1)^n(q;q^2)_n}{(-x)_{n+1}(-q/x)_n}=m(x,q,-1).
\end{equation*}
For the dual of second type, iterating the functional equation leads to
{\allowdisplaybreaks \begin{align}
f(a):=\Big ( 1+\frac{1}{a}\Big )&\sum_{n=0}^{\infty}\frac{(-aq,-q/a;q)_nq^{n+1}}{(q;q^2)_{n+1}}\label{equation:ABII-6.3.1-dualtypeII}\\
&=-m(a,q,-1)+\frac{j(-a;q)}{J_{1,2}}m(a^2,q^2,-1) -a\frac{J_4^3}{J_2^3}\frac{j(a;q)j(qa^2;q^2)}{j(a^4;q^4)},\notag
\end{align}}%
where the functional equation is
\begin{equation*}
f(qa)+1+af(a)=\frac{1}{a}\frac{j(-a;q)}{{J}_{1,2}}.
\end{equation*}
\begin{theorem}
Identity (\ref{equation:ABII-6.3.1-dualtypeII}) is true.
\end{theorem}
\begin{proof}
The method of \cite{L} yields,
\begin{equation*}
\Big ( 1+\frac{1}{a}\Big )\sum_{n=0}^{\infty}\frac{(-aq,-q/a;q)_nq^{n+1}}{(q;q^2)_{n+1}}=qf_{2,2,1}(q^3,-q^2a,q)/J_{1,2}.
\end{equation*}
The result then follows from Proposition \ref{proposition:f221}.  
\end{proof}
Equation (\ref{equation:ABII-6.3.1-dualtypeII}) is also the dual for \cite[Entry $5.4.3$]{ABII}, \cite[Entry $6.4.6$]{ABII}, also \cite[p. 4]{RLN} :
For $a\ne0$,
\begin{align}
\Big ( 1+\frac{1}{a}\Big )\sum_{n=0}^{\infty}{}^*\frac{(-1)^n(-aq,-q/a)_n}{(q;q^2)_{n+1}}
&=\frac{1}{2}\sum_{n=0}^{\infty}(-1)^n (a^n+a^{-n-1})q^{n(n+1)/2}\label{equation:ABII-5.4.3}.
\end{align}

\section{acknowledgements} We would like to thank Dean Hickerson for his help in finding several of the identities in Proposition \ref{proposition:eulerian-mxqz-prop} and for his helpful suggestions.  We would also like to thank the referees for their thorough reading of the manuscript and their detailed comments.

\section{concluding remarks}
If we are given one type of identity and the shifts $q\rightarrow q^{-1}$ and $n\rightarrow -n$ make sense, then the techniques in this paper are very effective in determining the structure of the dual identity.   We point out that for the two fifth order functions $\chi_0(q)$ and $\chi_1(q)$ that the multiplicities of the Appell--Lerch sums and partial theta functions do not agree.    Although we do not have short proofs of identities (\ref{equation:mock-chi1-5th-dualB}), (\ref{equation:ABII-6.3.6-dualtypeII}), (\ref{equation:ABII-6.3.9-dualtypeII}), the identities are included for a sense of completeness.   The natural next step in developing the duality theory between Appell--Lerch sums and partial theta functions is to determine if there are finite versions of the mock theta functions which can simultaneously prove mock theta function identities as well as the corresponding partial theta function identities, and if so find them.

\end{document}